\documentclass[12pt]{amsart}

\usepackage{amsmath, amssymb, amsfonts, amsthm,amscd}
\usepackage{mathrsfs,bbm}
\usepackage{txfonts}
\usepackage[all]{xy}
\usepackage{nicefrac}

\usepackage{nicefrac}

\newtheorem{theorem}{Theorem}[section]
\newtheorem{lemma}[theorem]{Lemma}
\newtheorem{proposition}[theorem]{Proposition}
\newtheorem{definition}[theorem]{Definition}

\newtheorem{conjecture}[theorem]{Conjecture}

\def\text{\mbox}
\def\tilde{\widetilde}
\def\hat{\widehat}

  \renewcommand{\hat}{\widehat}
  \renewcommand{\tilde}{\widetilde}

    \DeclareMathOperator{\Gr}{Gr}
     \DeclareMathOperator{\GL}{GL}    
     \DeclareMathOperator{\SL}{SL}   
      \DeclareMathOperator{\N}{N}

\newcommand{\beqn}{\begin{equation}}
\newcommand{\eeqn}{\end{equation}}

 \newcommand{\I}{\mathcal{I}}
 
       \title[Counter Example to a Strong Matroid Minor Conjecture]       
       {Counter Example to a Strong Matroid Minor Conjecture}
           
\date{}

\author{Shrawan Kumar}

\address{S. Kumar: Department of Mathematics, University of North Carolina, Chapel Hill, NC 27599-3250, USA}
\email{shrawan@email.unc.edu} 
\begin{document}

\maketitle{}

\noindent
{\bf Abstract:} The main result of this note asserts that a strong form of the Matroid Minor Conjecture due to J. Draisma is not true, i.e., there exist properly ascending chains of $S_\infty$-stable ideals in the affine coordinate ring of the {\it affine infinite Grassmannian}, where $S_\infty$ is the infinite symmetric group. In fact, we explicitly construct such an ascending chain. His conjectures on topological noetherian property for the affine infinite Grassmannian remain open though. 

\section{Introduction} Let $k$ be any (including finite)  field. For positive integers $n, p$, let $V_{n,p}$ be the vector space over  $k$ with basis 
$$\{x_{-n}, x_{-n+1}, \dots, x_{-1}, x_1, x_2, \dots, x_p\}.$$
Let $\{x^*_{-n}, x^*_{-n+1}, \dots, x^*_{-1}, x^*_1, x^*_2, \dots, x^*_p\}$ be the dual basis of the dual vector space $V_{n,p}^*$.

Let $\Gr(p, V^*_{n,p})$ be the Grassmannian of $p$-planes in $V^*_{n,p}$. Then, we have the Pl\"ucker embedding: 
\[\iota: \Gr(p, V^*_{n,p})\hookrightarrow \mathbb{P}\left(\wedge^p(V^*_{n,p})\right), \,\, A \mapsto \wedge^p(A), \,\,\text{for $A \in  \Gr(p, V^*_{n,p})$}.\]
Let $\tilde{\Gr}(p, V^*_{n,p})$ be the corresponding affine cone (under the above embedding). Thus, we get an embedding
\[\tilde{\Gr}(p, V^*_{n,p}) \hookrightarrow \wedge^p(V^*_{n,p}).\]
 This makes $\tilde{\Gr}(p, V^*_{n,p})$ a closed irreducible (affine) subvariety of $ \wedge^p(V^*_{n,p})$. 
 
 Define the surjective maps 
\[\xi= \xi^{n+1, p}: \wedge^p(V^*_{n+1,p}) \twoheadrightarrow  \wedge^p(V^*_{n,p})\]
induced from the standard inclusion of the bases and 
\[\theta=\theta^{n, p+1}:  \wedge^{p+1}(V^*_{n,p+1}) \twoheadrightarrow  \wedge^{p}(V^*_{n,p}),\,\,\, \omega \mapsto  i_{x_{p+1}}\omega,\]
where $i$ is the interior multiplication.  
Now, the maps
$\theta^{n+1, p+1}$ and $\xi^{n+1, p+1}$
 induce the restriction maps which are surjective:
\begin{equation*} \label{eqnn2}
\tilde{\Gr}(p, V^*_{n+1,p})\xleftarrow{\tilde{\theta}}  \tilde{\Gr}(p+1, V^*_{n+1,p+1})     \xrightarrow{\tilde{\xi}} \tilde{\Gr}(p+1, V^*_{n,p+1}).
\end{equation*}
Define a partial order $\leq$ on $\mathbb{Z}_{\geq 1}\times \mathbb{Z}_{\geq 1}$ by
$(n,p) \leq (m, q)$ if $n\leq m$ and $p\leq q$. 
The above maps  $\tilde{\theta}$ and $\tilde{\xi}$ give rise to a surjective map between the affine varieties (see $\S$2 for more details):
\begin{equation*} \label{eqnn4}
\tilde{\Gr}(q, V^*_{m, q})\to  \tilde{\Gr}(p, V^*_{n,p})\,\,\,\text{for  $(n,p) \leq (m, q)$}.
\end{equation*}
Define the affine schemes:
\begin{equation*}
\tilde{\Gr}(\infty/2, V^*_\infty):= \varprojlim_{(n, p)}\,  \tilde{\Gr}(p, V^*_{n,p}),\,\,\text{and}\,\, 
 \wedge^{\infty/2} (V^*_\infty):= \varprojlim_{(n, p)}\,  \wedge^p(V^*_{n,p}).
\end{equation*}
 Define the group 
$$\GL(\infty):= \varinjlim_{(n, p)}\, \GL(V_{n, p}),\,\,\,\text{
and its subgroup}\,\,
\N(\infty):= \varinjlim_{(n, p)}\, \N(V_{n, p}),$$
where  $\N(V_{n, p})$ denotes the normalizer of the standard maximal torus in $\SL(V_{n, p})$. 
 The standard action of $\GL(V_{n, p})$ on  $\wedge^p(V^*_{n,p})$ gives rise to  an action of $\GL(\infty)$ on $\tilde{\Gr}(\infty/2, V^*_\infty)$
 (cf. $\S$2 for more details).  
Define the infinite symmetric group  
\begin{equation*}  S_\infty :=  \varinjlim_{n}\, S_n, \,\,\,\text{where $S_n$ is the symmetric group on the symbols 
 $\{-n, -(n-1), \dots, -1, 1, 2, \dots, n\}$}. 
\end{equation*}
Then, $S_\infty$ canonically embeds in $\GL(\infty)$ via the permutation matrices.
\vskip1ex

Jan Draisma kindly told us his following conjecture.
\begin{conjecture} Let $k$ be any field.  
The affine infinite Grassmannian  $\tilde{\Gr}(\infty/2, V^*_\infty)$  is topologically $S_\infty$-noetherian, i.e., every descending chain of  
$S_\infty$-stable Zariski-closed subsets of  $\tilde{\Gr}(\infty/2, V^*_\infty)$
stabilizes. 
\vskip1ex

A slightly weaker form of the conjecture states that  $\tilde{\Gr}(\infty/2, V^*_\infty)$ is  topologically $\N(\infty)
$-noetherian.
\vskip1ex
A stronger form of the conjecture states that any ascending chain of $S_\infty$-stable ideals in the affine coordinate ring of \,
 $\tilde{\Gr}(\infty/2, V^*_\infty)$ stabilizes. 
\end{conjecture}

The main result of this note is the following (cf. Theorem \ref{thm6}).
\begin{theorem} The strong form of the above conjecture is false, i.e., there exists an ascending chain of $S_\infty$-stable ideals in the affine coordinate ring of 
 $\tilde{\Gr}(\infty/2, V^*_\infty)$ which does not  stabilize. In fact, we give such an example explicitly. 
 \end{theorem}
 
Before we can explain the significance of Draisma's Conjecture \ref{conj2.8} to some  important results in Graph Theory, we need to 
 briefly explain some of the very significant results in  {\it Matroid Minors Theory}.

We begin by recalling the following conjecture due to Rota \cite{Ro}.
\begin{conjecture}\label{conj2.1}
 For each finite field $k$, there
are, up to isomorphism, only finitely many excluded
minors for the class of $F$-representable matroids.
\end{conjecture}

Rota's conjecture is reminiscent of the classical Generalized
Kuratowski Theorem \cite{Kur}. 
As part of the {\it Graph Minors Project},
Neil Robertson and Paul Seymour were able to further generalize the Generalized
Kuratowski Theorem to obtain the WQO ({\it Well Quasi Ordering}) Theorem
 stated below (cf. \cite{RS}).  Their results were published
in a series of twenty three journal papers totaling
more than 700 pages from 1983 to 2004. 
Diestel, in his book on graph theory \cite{Di}, says that
this theorem dwarfs any other result in graph
theory and may doubtless be counted among the
deepest theorems that mathematics has to offer.

\begin{theorem}\label{thm2.1}
 (WQO). Each
minor-closed class of graphs has only finitely many
excluded minors.

Equivalently,  in any infinite set $S$ of graphs, there must be a pair of graphs one of which is a minor of the other.
\end{theorem}

Then, 
Robertson and Seymour proposed ideas for
extending their Graph Minors Project to matroids. The challenge was taken up by Jim Geelen, Bert Gerards, and Geoff Whittle.  Though  it is
not true that the WQO Theorem extends to all
matroids. However, 
after extensive work for several years,  they (Geelen et al.) announced the following slightly weaker  theorem (cf. \cite[Theorem 6]{GGW}) significantly extending the WQO theorem to matroids.

\begin{theorem}  \label{thm2.2}  (Matroid WQO Theorem). For each
finite field $k$ and each minor-closed class of $k$-representable matroids, there are only finitely
many $k$-representable excluded minors.
\end{theorem}

According to their article \cite{GGW}, to quote them: `We are now immersed in the lengthy task of
writing up our results. Since that process will take
a few years, we have written this article offering a
high-level preview of the proof.'
\vskip1ex

The following result (communicated to us by J. Draisma), which is fairly easy to prove, provides a direct  bridge between  Conjecture \ref{conj2.8} and Thorem \ref{thm2.2} once we observe that $\theta^{n, p+1}: \wedge^{p+1} (V_{n, p+1}^*) \to  \wedge^{p} (V_{n, p}^*)$ and  $\xi^{n+1, p}: \wedge^{p} (V_{n+1, p}^*) \to  \wedge^{p} (V_{n, p}^*)$,
at the level of matroids, correspond to contraction and deletion respectively.

\begin{theorem}  \label{thm2.13}Let $k$ be a finite field and assume that the set of $k$-points of the affine infinite Grassmannian  $\tilde{\Gr}(\infty/2, V^*_\infty)$, equipped with the Zariski topology, is $S_\infty$-noetherian.
Then, matroids representable over $k$ are well quasi ordered by the minor order.
\end{theorem}

\noindent
{\bf Acknowledgements:} We are indebted to Jan Draisma for explaining to us his Conjecture \ref{conj2.8},  showing its significance via his Theorem \ref{thm2.13} to the Matroid WQO Theorem \ref{thm2.2}, providing all the references  \cite{Di}, \cite{GGW}, \cite{Kur}, \cite{RS}, \cite{Ro}. This work was completed while the author was visiting the 
Institut des Hautes \'Etudes Scientifiques (Bures-sur-Yvette, France) during the fall semester of 2023, hospitality of which is gratefully acknowledged. 
\section{Infinite Grassmannian and the Matroid Minor Conjecture}

The base field in this note is any (including finite)  field.

\begin{definition} \label{defA.1} {\rm For positive integers $n, p$, let $V_{n,p}$ be the vector space over  $k$ with basis 
$$\{x_{-n}, x_{-n+1}, \dots, x_{-1}, x_1, x_2, \dots, x_p\}.$$
Let $\{x^*_{-n}, x^*_{-n+1}, \dots, x^*_{-1}, x^*_1, x^*_2, \dots, x^*_p\}$ be the dual basis of the dual vector space $V_{n,p}^*$. 

  Define the linear maps
\[\eta= \eta^{n+1, p}: \wedge^p(V_{n,p}) \hookrightarrow  \wedge^p(V_{n+1,p})\]
induced from the standard inclusion of the bases
and
\[\beta= \beta^{n, p+1}:  \wedge^p(V_{n,p}) \hookrightarrow  \wedge^{p+1}(V_{n,p+1}),\,\,\, \omega \mapsto \omega\wedge x_{p+1}.\]
Dually, we get surjective maps 
\[\xi= \xi^{n+1, p}: \wedge^p(V^*_{n+1,p}) \twoheadrightarrow  \wedge^p(V^*_{n,p})\]
and 
\[\theta=\theta^{n, p+1}:  \wedge^{p+1}(V^*_{n,p+1}) \twoheadrightarrow  \wedge^{p}(V^*_{n,p}),\,\,\, \omega \mapsto  i_{x_{p+1}}\omega,\]
where $i$ is the interior multiplication. 
Let $\Gr(p, V^*_{n,p})$ be the Grassmannian of $p$-planes in $V^*_{n,p}$ (For generalities on Grassmannians, see \cite[$\S$III.2.7]{EH}.) . Then, we have the Pl\"ucker embedding: 
\[\iota: \Gr(p, V^*_{n,p})\hookrightarrow \mathbb{P}\left(\wedge^p(V^*_{n,p})\right), \,\, A \mapsto \wedge^p(A), \,\,\text{for $A \in  \Gr(p, V^*_{n,p})$}.\]
Let $\tilde{\Gr}(p, V^*_{n,p})$ be the corresponding affine cone (under the above embedding). Thus, we get an embedding
\[\tilde{\Gr}(p, V^*_{n,p}) \hookrightarrow \wedge^p(V^*_{n,p}),\]
where the image consists of the decomposable vectors (including the vector $0$). This makes $\tilde{\Gr}(p, V^*_{n,p})$ a closed irreducible (affine) subvariety of $ \wedge^p(V^*_{n,p})$. 
Now, the maps
\begin{equation} \label{eqn1}
\wedge^p (V_{n+1, p}^*)\xleftarrow{\theta^{n+1, p+1}} \wedge^{p+1} (V_{n+1, p+1}^*)\xrightarrow{\xi^{n+1, p+1}} \wedge^{p+1} (V_{n, p+1}^*)
\end{equation}
 induce the restriction maps which are surjective:
\begin{equation} \label{eqn2}
\tilde{\Gr}(p, V^*_{n+1,p})\xleftarrow{\tilde{\theta}}  \tilde{\Gr}(p+1, V^*_{n+1,p+1})     \xrightarrow{\tilde{\xi}} \tilde{\Gr}(p+1, V^*_{n,p+1}).
\end{equation}
To prove the existence of $\tilde{\theta}$, we can write 
$$v_1\wedge \dots \wedge v_{p+1} = v'_1\wedge \dots \wedge v'_p\wedge (v'_{p+1}+ \alpha x^*_{p+1}),\,\,\text{where $v'_i(x_{p+1})=0 \forall 1\leq i\leq p+1$},$$
for some $\alpha\in k$.  Thus, 
$$i_{x_{p+1}}(v_1\wedge \dots \wedge v_{p+1})=\pm \alpha v'_1\wedge \dots \wedge v'_p$$
and hence  $\theta$ induces the map  $\tilde{\theta}$ on the corresponding cones of Grassmannians. The existence of $\tilde{\xi}$ is trivial to see. 

Define a partial order $\leq$ on $\mathbb{Z}_{\geq 1}\times \mathbb{Z}_{\geq 1}$ by
\begin{equation} \label{eqn3}
(n,p) \leq (m, q)\,\,\,\text{if $n\leq m$ and $p\leq q$}. 
\end{equation}
We have a surjective map between the affine varieties induced from the maps $\tilde{\theta}$ and $\tilde{\xi}$:
\begin{equation} \label{eqn4}
\tilde{\Gr}(q, V^*_{m, q})\to  \tilde{\Gr}(p, V^*_{n,p})\,\,\,\text{for  $(n,p) \leq (m, q)$}.
\end{equation}

The above map is well defined since the following diagram is commutative:
\begin{equation}
\label{eqn5}
\xymatrix{
\tilde{\Gr}(p+1, V^*_{n+1,p+1})\ar[r]^{\tilde{\xi}^{n+1, p+1}} \ar[d]^{\tilde{\theta}^{n+1, p+1}} & \tilde{\Gr}(p+1, V^*_{n,p+1}) \ar[d]^{\tilde{\theta}^{n, p+1}} \\
 \tilde{\Gr}(p, V^*_{n+1,p})  \ar[r]^{\tilde{\xi}^{n+1, p}}&  \tilde{\Gr}(p, V^*_{n,p}).
}
\end{equation}

Define the affine schemes:
\begin{equation}
\label{eqn6} 
\tilde{\Gr}(\infty/2, V^*_\infty):= \varprojlim_{(n, p)}\,  \tilde{\Gr}(p, V^*_{n,p}),
\end{equation}
and 
\begin{equation}\label{eqn6.5} 
\wedge^{\infty/2} (V^*_\infty):= \varprojlim_{(n, p)}\,  \wedge^p(V^*_{n,p}).
\end{equation}
We call $\tilde{\Gr}(\infty/2, V^*_\infty)$ the {\it affine infinite Grassmannian}. 

Then, the corresponding affine coordinate rings are given by:
\begin{equation}
\label{eqn7} 
k[\tilde{\Gr}(\infty/2, V^*_\infty)]= \varinjlim_{(n, p)}\,  k[\tilde{\Gr}(p, V^*_{n,p})],
\end{equation}
and
\begin{equation}
\label{eqn7.5} 
k[\wedge^{\infty/2} (V^*_\infty)]= \varinjlim_{(n, p)}\,  k[\wedge^p(V^*_{n,p})]=  \varinjlim_{(n, p)}\,  S^\bullet(\wedge^p(V_{n,p})),\end{equation}
where $S^\bullet$ is the symmetric algebra.

The multiplicative group $k^*$ acts on $\wedge^p(V^*_{n,p})$ via multiplication by $z\in k^*$. Clearly, this
$k^*$-action preserves $\tilde{\Gr}(p, V^*_{n,p})$. Moreover, $\tilde{\xi}$ and $\tilde{\theta}$ both commute with this $k^*$-action. Thus, we get a $k^*$-action on $\tilde{\Gr}(\infty/2, V^*_\infty)$. Further, $\tilde{\xi}^{n+1, p+1}$ commutes with the 
standard $\GL(V_{n, p+1})$-actions (considering $\GL(V_{n, p+1})$ as canonically embedded in $\GL(V_{n+1, p+1})$)
and $\tilde{\theta}^{n+1, p+1}$ commutes with the standard $\GL(V_{n+1, p})$-actions (again considering $\GL(V_{n+1, p})$ as canonically embedded in $\GL(V_{n+1, p+1})$). In particular, $\tilde{\xi}^{n+1, p+1}$ commutes with the 
$\N(V_{n, p+1})$-actions and $\tilde{\theta}^{n+1, p+1}$ commutes with the $\N(V_{n+1, p})$-actions, where $\N(V_{n+1, p})$ denotes the normalizer of the standard maximal torus in $\SL(V_{n+1, p})$. 

Define the group 
$$\GL(\infty):= \varinjlim_{(n, p)}\, \GL(V_{n, p}),$$
and its subgroup
$$\N(\infty):= \varinjlim_{(n, p)}\, \N(V_{n, p}).$$
Then, $\GL(\infty)$ acts on $\tilde{\Gr}(\infty/2, V^*_\infty)$ as follows. Take $g\in \GL(V_{n, p})$. Now, 
\begin{equation}
\label{eqn8} 
\tilde{\Gr}(\infty/2, V^*_\infty):= \varprojlim_{(m,q)\geq (n, p)}\,  \tilde{\Gr}(q, V^*_{m,q}).
\end{equation}
Each variety on the right is acted upon by $\GL(V_{n, p})$ and all the maps $\tilde{\xi}$ and $\tilde{\theta}$ are $\GL(V_{n, p})$-equivariant maps. Thus, $g\in \GL(V_{n, p})$ acts on $\tilde{\Gr}(\infty/2, V^*_\infty)$ for any pair $(n,p)$. These actions clearly combine to give an action of  $\GL(\infty)$ on  $\tilde{\Gr}(\infty/2, V^*_\infty)$. 

The standard action of $\GL(V_{n, p})$ on $\wedge^p(V^*_{n,p})$ commutes with the 
above $k^*$-action and hence so is the $k^*$-action  on $\tilde{\Gr}(p, V^*_{n,p})$ commutes with the standard $\GL(V_{n,p})$-action. Thus, we get a  $k^*$-action on  $\tilde{\Gr}(\infty/2, V^*_\infty)$ commuting with the action of $\GL(\infty)$. 

Define a bijection
$$ \mu: -\mathbb{N} \sqcup \mathbb{N} \to \mathbb{N}, -n \mapsto 2n,\,\, p\mapsto 2p-1,\,\,\text{for $n, p\in \mathbb{N}$},$$
where $\mathbb{N}$ is the set of positive integers $\{1, 2, \dots\}$. This gives rise to a {\it well order} on $ -\mathbb{N} \sqcup \mathbb{N} $ transporting the standard well order on $\mathbb{N}$ via $\mu$. Write a basis of $\wedge^p(V_{n,p})$ as follows:
$$x_{i_1}\wedge \dots \wedge x_{i_p},\,\,\text{where $i_j\in \{-n, -(n-1), \dots,-1, 1, \dots, p\}$}$$
so that $i_1< i_2 < \dots < i_p$ in the above well order. Now, define
$${\bf x_i}:= x_{i_1}\wedge \dots \wedge x_{i_p}< x_{j_1}\wedge \dots \wedge x_{j_p}$$
 under the lexicographic order reading from the left.}
 \end{definition}
The following lemma is clear.
\begin{lemma} \label{lem1}The above order on the basis of $\wedge^p(V_{n,p})$ is a well order.
\end{lemma}

\begin{lemma}\label{lem2} Under the embedding $\beta:  \wedge^p(V_{n,p}) \hookrightarrow  \wedge^{p+1}(V_{n,p+1}),\, \omega \mapsto \omega\wedge x_{p+1}$, the above well ordering on the basis of $ \wedge^{p+1}(V_{n,p+1})$ restricts to the well ordering on the basis of  $\wedge^p(V_{n,p})$. 
\end{lemma}
\begin{proof} First,  let $x_{i_1}\wedge \dots \wedge x_{i_p}< x_{j_1}\wedge \dots \wedge x_{j_p}.$ Then, we claim that 
 \begin{equation} 
 \label{eqn9} x_{i_1}\wedge \dots \wedge x_{i_p}\wedge x_{p+1}< x_{j_1}\wedge \dots \wedge x_{j_p}\wedge x_{p+1}:
 \end{equation}
 
 Choose the largest $\ell \geq 0$ such that $i_1=j_1, \dots, i_\ell =j_\ell$. If $\mu(p+1)< \mu (i_\ell)$, then clearly the equation \eqref{eqn9} is true. So, assume that $\mu(p+1)> \mu (i_\ell)$. If $\mu(p+1)> \mu(j_{\ell +1}) >  \mu (i_{\ell +1})$,  then again clearly the equation \eqref{eqn9} is true.  So, assume that $\mu(i_\ell) =\mu(j_\ell) < \mu(p+1) <  \mu (j_{\ell +1})$. If $\mu(i_{\ell +1}) > \mu(p+1)$,  then again the equation \eqref{eqn9} is true. So, finally assume that $\mu(i_{\ell +1}) < \mu(p+1)$ and $\mu(j_{\ell +1}) > \mu(p+1)$.  Since $\mu(i_{\ell +1}) < \mu(p+1)$,   the equation \eqref{eqn9} is true since $ \mu(j_\ell)  < \mu(p+1) <  \mu (j_{\ell +1})$.

 Conversely, if 
 \begin{equation} 
 \label{eqn10} x_{i_1}\wedge \dots \wedge x_{i_p}\wedge x_{p+1}< x_{j_1}\wedge \dots \wedge x_{j_p}\wedge x_{p+1},
 \end{equation} 
 then we assert that $ x_{i_1}\wedge \dots \wedge x_{i_p}< x_{j_1}\wedge \dots \wedge x_{j_p}:$ For, otherwise, assume that 
 $ x_{i_1}\wedge \dots \wedge x_{i_p}> x_{j_1}\wedge \dots \wedge x_{j_p}$. This implies  $x_{i_1}\wedge \dots \wedge x_{i_p}\wedge x_{p+1}> x_{j_1}\wedge \dots \wedge x_{j_p}\wedge x_{p+1} $,  contradicting the equation \eqref{eqn10}. This proves the lemma. 
 \end{proof}
 The proof of the following lemma is clear. 
 \begin{lemma}\label{lem3} Under the embedding $\eta:  \wedge^p(V_{n,p}) \hookrightarrow  \wedge^{p}(V_{n+1,p}),\, \omega \mapsto \omega$, the above well ordering on the basis of $ \wedge^{p}(V_{n+1,p})$ restricts to the well ordering on the basis of  $\wedge^p(V_{n,p})$. 
\end{lemma} 
 We recall the following definition from \cite[$\S$3.3]{AH}.
 \begin{definition} {\rm Let $\leq$ be a well ordering on a countable set $X$. Define the induced {\it lexicographic ordering} $\leq^o$ on the set $X^o$ of commuting monomials with terms from $X$ as follows:
 \begin{equation}\label{eqn11} 
 {\bf x} := x_1^{m_1}\dots x_a^{m_a} \leq^o {\bf y}:= x_1^{n_1}\dots x_a^{n_a} \Leftrightarrow  (m_1, \dots, m_a) \leq (n_1, \dots, n_a)  \end{equation}
 lexicographically from the left, where $x_i\in X$, $x_1 < x_2 < \dots < x_a$  and $m_i, n_i\in \mathbb{Z}_{\geq 0}$.  Then, $\leq^o$ is a term ordering on $X^o$ (cf. \cite[Example2.5]{AH}. Let a group $G$ act on $X$. Define a quasi-ordering $|_G$ on $X^o$ by
 $${\bf x} |_G {\bf y}  \Leftrightarrow \exists \sigma\in G \,\text{and} \,\, {\bf z}\in X^o: (\sigma {\bf x}) {\bf z} = {\bf y}.
 $$}
 \end{definition}
 Define the infinite symmetric group  $S_\infty :=  \varinjlim_{n}\, S_n$, where $S_n$ is the symmetric group on the symbols 
 $\{-n, -(n-1), \dots, -1, 1, 2, \dots, n\}$. Then, $S_\infty$ is canonically embedded as a subgroup of $\GL(\infty)$ obtained via the permutation matrices. 
 
 \begin{conjecture} (due to J. Draisma) \label{conj2.8} Let $k$ be any field.  
 	The affine infinite Grassmannian  $\tilde{\Gr}(\infty/2, V^*_\infty)$  is topologically $S_\infty$-noetherian, i.e., every descending chain of  
 	$S_\infty$-stable Zariski-closed subsets of  $\tilde{\Gr}(\infty/2, V^*_\infty)$
 	stabilizes. 
 	\vskip1ex
 	
 	A slightly weaker form of the conjecture states that  $\tilde{\Gr}(\infty/2, V^*_\infty)$ is  topologically $\N(\infty)
 	$-noetherian.
 	\vskip1ex
 	A stronger form of the conjecture states that any ascending chain of $S_\infty$-stable ideals in the affine coordinate ring of \,
 	$\tilde{\Gr}(\infty/2, V^*_\infty)$ stabilizes. 
 \end{conjecture}
 
 \section{A counterexample to a stronger form of the Matroid Minor Conjecture}
 
 We begin first by disproving the strong form of Matroid Minor Conjecture for $\wedge^{\infty/2}(V^*_\infty)$. 
 \begin{proposition} \label{thm 4} The ring 
 $$R:=k[\wedge^{\infty/2}(V^*_\infty)] =  \varinjlim_{(n, p)}\,  S^\bullet(\wedge^p(V_{n,p}))\,\,\,\text{(cf. the identity \eqref{eqn7.5})} 
 $$ is not noetherian with respect to the $S_\infty$-stable ideals, i.e., there exists a strictly increasing sequence of $S_\infty$-stable ideals of $R$: 
 $$I_1 \subsetneq I_2 \subsetneq I_3 \subsetneq \dots .$$
 \end{proposition}
 \begin{proof} Using \cite[Lemma3.14]{AH}, it suffices to show that $|_{S_\infty}$ is {\it not} well-quasi-ordering. Consider the monomials 
  \begin{align}\label{eqn12} 
 \{a_nb_n\}_{n\geq 3}, &\,\,\,\text{where $a_n:= x_{-2}\wedge x_{2}\wedge x_3\wedge \dots \wedge x_n$}\notag\\
  &\text{and  $b_n:= x_{-(n-1)}\wedge x_{-(n-2)}\wedge \dots \wedge x_{-1}\wedge x_1\wedge x_{n+1}$}.
  \end{align} 
  We claim that $a_nb_n\not{|}_{S_\infty} a_mb_m$ for $n\neq m\geq 3$:
  
  If not, let $\sigma\in S_\infty$ be such that $\sigma (a_n) \cdot\sigma (b_n)= a_mb_m$ (observe that they both are degree $2$ monomials), i.e., 
  \begin{align*} \label{eqn13} &\left(x_{\sigma(-2)}\wedge x_{\sigma (2)}\wedge x_{\sigma (3)}\wedge \dots \wedge x_{\sigma (n)}\right)\notag\\
  &\cdot  \left(
  x_{\sigma(-(n-1))}\wedge x_{\sigma(-(n-2))}\wedge \dots \wedge x_{\sigma (-1)}\wedge x_{\sigma (1)}\wedge x_{\sigma (n+1)} \right)\notag\\
  &=\left(x_{-2}\wedge x_{2}\wedge x_3\wedge \dots \wedge x_m\right) \cdot   \left(x_{-(m-1)}\wedge x_{-(m-2)}\wedge \dots \wedge x_{-1}\wedge x_1\wedge x_{m+1}\right). 
  \end{align*}
  This gives 
  \begin{align*}  \sigma \left(\{-2, 2, 3, \dots, n, n+1,  \dots \}\right) \setminus  \sigma\left(\{-(n-1), -(n-2), \dots, -1, 1, n+1, n+2, \dots  \}
  \right)\\=\begin{cases}   
  \{-2, 2, 3, \dots, m, m+1, \dots \} \setminus  \{-(m-1), -(m-2), \dots, -1, 1, m+1, m+2, \dots  \} \,\,\text{or}\\
    \{-(m-1), -(m-2), \dots, -1, 1, m+1, m+2, \dots  \}  \setminus   \{-2, 2, 3, \dots, m, m+1, \dots \}.
  \end{cases}
  \end{align*}  
  The above equation is equivalent to the following:
   \begin{align*} \sigma \left(\{2, 3, \dots, n \}\right) =\begin{cases}
 \{2, 3, \dots, m \}\,\,\text{or}\\
    \{-(m-1), -(m-2), \dots, -3, -1, 1 \}.
   \end{cases}
  \end{align*}   
  The left side of the above equation has cardinality $n-1$, whereas the ride side has cardinality $m-1$. This is a contradiction since $n\neq m$ by assumption. Thus, the infinite set $A=  \{a_nb_n\}_{n\geq 3}$ is an anti-chain under  $|_{S_\infty}$ (cf. \cite[Page 5173]{AH}). This proves the proposition.
    \end{proof}  
    
    Consider the ordered basis 
    $$ x_{-1}, x_{-2}, \dots, x_{-(m-1)}, x_1, x_2, \dots, x_{m+1}\,\,\text{of $ V_{m-1, m+1}$}.$$
    We abbreviate $m+1$ by $p$. Consider the standard maximal parabolic  subgroup $P(m-1)$ of $\SL(2m)$ (with respect to the above ordered basis) obtained by deleting the $(m-1)$-th node from the Dynkin diagram of $\SL(2m)$. Thus, 
    $P= P(m-1)$ is the stabilizer of the line $[x_1^*\wedge   x_2^*\wedge \dots \wedge x_p^*] \in \mathbb{P}( V_{m-1, p}^*)$. Consider the opposite unipotent radical $U^-= U^-_{m-1}$ of $P$. Thus, for any $g\in U^-$,
    \begin{align*} &g^{-1}(x_{-j}) = x_{-j}+ \sum_{i=1}^p \alpha_i^j(g) x_i, \,\,\,\text{for $1\leq j \leq m-1$, and}\\
    &g^{-1}(x_{i}) = x_{i}, \,\,\,\text{for $1\leq i \leq p$,}
     \end{align*}  
    for some $\alpha_i^j(g)\in k$. In fact,  $U^-$ is characterized by the above, where we allow  $\alpha_i^j(g)$ to vary  over $k$.
  The action of $U^-$ on the dual basis is given as follows (for $g\in U^-$):
    \begin{align*} &g(x_{-j}^*) = x_{-j}^*, \,\,\,\text{for $1\leq j \leq m-1$, and}\\
    &g(x_{i}^*) = x_{i}^* + \sum_{j=1}^{m-1} \alpha_i^j(g) x_{-j}^*, \,\,\,\text{for $1\leq i \leq p$,}
     \end{align*}   
     Thus, for $g\in U^-$,
     \begin{equation} \label{eqn14}
     g\left(x_1^*\wedge   x_2^*\wedge \dots \wedge x_p^*\right)= \left(x_1^*+\sum_{j_1=1}^{m-1}\,\alpha_1^{j_1}(g)x_{-j_1}^*\right)\wedge \dots \wedge    \left(x_p^*+\sum_{j_p=1}^{m-1}\,\alpha_p^{j_p}(g)x_{-j_p}^*\right).
     \end{equation} 
     With the notation as above, we have the following lemma.  
     \begin{lemma} \label{lem5} For $g\in U^-$ and ${\bf x}= x_{-n_1}\wedge \dots \wedge x_{-n_q}\wedge x_{d_1} \wedge \dots \wedge x_{d_{p-q}}$, where $ 0\leq q\leq m-1$, $0< n_1 < \dots < n_q \leq m-1$ and $0< d_1 < \dots < d_{p-q}\leq p$, we have the following identity:
     \begin{align*}  &g\left(x_1^*\wedge   x_2^*\wedge \dots \wedge x_p^*\right) \left(x_{-n_1}\wedge \dots \wedge x_{-n_q}\wedge x_{d_1} \wedge \dots \wedge x_{d_{p-q}} \right)\\
     &= \pm \det \left(\alpha^{n_i}_{m_j}\right)_{1\leq i\leq q; \,\,m_j\in \{1, 2, \dots, \hat{d_1}, \dots, \hat{d_{p-q}}, \dots, p\}}.
     \end{align*} 
  \end{lemma}
  \begin{proof} 
   
  \begin{equation*}g\left(x_1^*\wedge   x_2^*\wedge \dots \wedge x_p^*\right)   {\bf x}  =\det 
\begin{pmatrix}
\alpha_1^{n_1}(g), &\ldots, &\alpha_1^{n_q}(g), & 0, & 0, &\ldots, &0 \\
\alpha_2^{n_1}(g), &\ldots, &\alpha_2^{n_q}(g), & 0, & 0, &\ldots, &0 \\
\vdots& &\vdots &\vdots & \vdots &&\vdots  \\
\alpha_{d_1}^{n_1}(g), &\ldots, &\alpha_{d_1}^{n_q}(g), & 1, & 0, &\ldots, &0 \\
\vdots& &\vdots &\vdots & \vdots &&\vdots  \\
\alpha_{d_2}^{n_1}(g), &\ldots, &\alpha_{d_2}^{n_q}(g), & 0, & 1, &\ldots, &0 \\
\vdots& &\vdots &\vdots & \vdots &&\vdots  \\
\alpha_{d_{p-q}}^{n_1}(g), &\ldots, &\alpha_{d_{p-q}}^{n_q}(g), & 0, & 0, &\ldots, &1 \\
\vdots& &\vdots &\vdots & \vdots &&\vdots  \\
\alpha_{{p}}^{n_1}(g), &\ldots, &\alpha_{{p}}^{n_q}(g), & 0, & 0, &\ldots, &0 
\end{pmatrix}.
    \end{equation*}
    By moving $d_1, d_2, \dots, d_{p-q}$-th rows of the above matrix to the end of the matrix, we get (from the above equation), 
     \begin{equation*}   
  g\left(x_1^*\wedge   x_2^*\wedge \dots \wedge x_p^*\right)   {\bf x}  =\pm \det \left(\alpha^{n_i}_{m_j}\right)_{1\leq i\leq q; \,\,m_j\in \{1, 2, \dots, \hat{d_1}, \dots, \hat{d_{p-q}}, \dots, p\}}.
  \end{equation*} 
    This proves the lemma.
     \end{proof}
  Let $\tilde{\Gr}(\infty/2, V^*_\infty) \hookrightarrow \wedge^{\infty/2}(V_\infty^*)$ be the Pl\"ucker embedding  induced from the  Pl\"ucker embeddings   $\tilde{\Gr}(p, V^*_{n,p})  \hookrightarrow \wedge^p(V^*_{n, p})$ 
  and let $\I \subset k[\wedge^{\infty/2}(V_\infty^*)] $ be the ideal generated by the Pl\"ucker relations, i.e., the ideal of the subscheme $\tilde{\Gr}(\infty/2, V^*_\infty)$ embedded in  $\wedge^{\infty/2}(V_\infty^*)$  via the Pl\"ucker embedding (cf. \cite[$\S$III.2.7]{EH}). Consider the sequence of ideals
  in $k[\wedge^{\infty/2}(V_\infty^*)] $:
  \begin{equation} \label{eqn15} 
  \I_n:= \left\langle S_\infty(a_3\cdot b_3), \dots, S_\infty(a_n\cdot b_n) \right\rangle +\I,
  \end{equation} 
  where $a_i, b_i$ are defined by the equation \eqref{eqn12} and $S_\infty (a_i\cdot b_i)$ denotes the collection $\{\sigma(a_i)\cdot \sigma(b_i)\}_{\sigma \in S_\infty}$. 
  
  We have the following main theorem of this note.
  \begin{theorem} \label{thm6}The above ideals satisfy:
  $$\I_3 \subsetneq \I_4 \subsetneq  \I_5 \subsetneq \dots .$$
  In particular, the ring $\mathcal{R}:=k[\tilde{\Gr}(\infty/2, V^*_\infty)]$ is not noetherian with respect to the $S_\infty$-stable ideals. 
  Thus, the stronger form of the Matroid Minor Conjecture (cf. Conjecture \ref{conj2.8}) is false.
  \end{theorem} \label{thm3.1}
  \begin{proof} Fix $\ell >3$. It suffices to show that $a_\ell\cdot b_\ell \notin \I_n$ for any $n< \ell$. This is equivalent to proving that for any $\sigma_i\in  k[S_\infty]$, where  $ k[S_\infty]$  is the group algebra of $S_\infty$, 
  $$
  a_\ell\cdot b_\ell - \sum_{i=3}^{\ell-1}\, \sigma_i (a_i\cdot b_i)\,\text{does not vanish identically on $\tilde{\Gr}(\infty/2, V^*_\infty)$}.$$ 
  Write 
  $$\sigma_i = \sum_k\, z^k_i \sigma^k_i\,\,\text{(a finite sum) for some $z^k_i\in k$ and $\sigma^k_i\in S_\infty$}.$$
  Take $m\geq \ell$ large enough so that each $\sigma_i^k$ (with nonzero $z^k_i$) is a permutation of the basis of $V_{m-1, p:=m+1}$. We want to show that 
   \begin{align} \label{eqn16}   
   a_\ell\cdot b_\ell - &\sum_{i=3}^{\ell-1}\, \sum_k \, z_i^k\sigma_i^k (a_i\cdot b_i)\,\text{thought of as a function on}\notag\\
   &     \wedge^p(V_{m-1, p}^*) \,\text{does not vanish identically on $U_m^-$}. 
    \end{align}   
  Now,
   \begin{equation*}
\label{eqn5}
\xymatrix{
S^\bullet \left(\wedge^p(V_{m-1, p})\right) \hookrightarrow \ar[d] & k\left[\wedge^{\infty/2}(V_\infty^*)\right]
 \ar[d] \\
 k\left[\tilde{\Gr}(p, V^*_{m-1,p})\right]  \hookrightarrow &  k\left[\tilde{\Gr}(\infty/2, V^*_\infty)\right], 
}
\end{equation*}   
 where $S^\bullet$ denotes the symmetric algebra and both the vertical maps are surjective.   
   Following the notation in Lemma \ref{lem5}, we rewrite 
 $$   \det \left(\alpha^{n_i}_{m_j}\right)_{1\leq i\leq q; \,\,m_j\in \{1, 2, \dots, \hat{d_1}, \dots, \hat{d_{p-q}}, \dots, p\}}= \det\left(\alpha^{n_1, \dots, n_q}_{m_1, \dots, m_q}\right).$$
 Assume, if possible, that   
   \begin{equation} \label{eqn17}
     \left(g\left(x_1^*\wedge   x_2^*\wedge \dots \wedge x_p^*\right)\right)  \left(a_\ell\cdot b_\ell - \sum_{i=3}^{\ell-1}\, \sum_k\, z^k_i\sigma^k_i (a_i\cdot b_i)\right)=0,\,\,\text{for all $g\in U^-_m$}.
     \end{equation} 
     By Lemma \ref{lem5}, we get 
     $$  g\left(x_1^*\wedge   x_2^*\wedge \dots \wedge x_p^*\right)  \left(a_\ell\cdot b_\ell\right)= \alpha_1^2 \det \left(\alpha^{1, \dots, \ell-1}_{2, 3, \dots, \ell}\right).$$  
  Now, if $  g\left(x_1^*\wedge   x_2^*\wedge \dots \wedge x_p^*\right)  \left(\sigma^k_i(a_i\cdot b_i)\right)$ has any nonzero contribution to the above term,  considering the action of $k^*$ on each of the variables $\alpha^i_j$ for a fixed $i$ (and any $j$) and similarly for $k^*$-action on $\alpha^i_j$ for a fixed $j$ (and any $i$) by the same character, we should have (by Lemma \ref{lem5}):
  $$ \sigma^k_i(A) = \{-C, -2, \ell+1, \dots, p, C'\},$$
  where $A:=\{-2, 2, 3, \dots, p\} $
  and 
  $$ \sigma^k_i(B) = \{-D, -2, \ell+1, \dots, p, D'\},$$
  for some $C, C', D, D'$ satisfying the following: 
  $$ \{1, \hat{2}, 3, \dots, \ell-1\}= C\sqcup D \,\,\text{and}\,\,  \{1, 2, 3, \dots, \ell\}= C'\sqcup D',$$
  where $ B:= \{-(i-1), -(i-2), \dots, -1, 1, i+1, i+2, \dots, p\}$.

  Setting $c=|C|$, we get 
  $$|D|= \ell-2-c, \,|C'|= -(p-\ell+c+1)+p= \ell-c-1,\,\text{and}\,\,   |D'|= -(p-\ell+1+\ell-2-c)+p= c+1.$$ 
  Then,
   \begin{equation} \label{eqn18}  \sigma^k_i(A\setminus B) = \sigma^k_i(\{2, 3, \dots, i\})
   \end{equation}
   and 
    \begin{equation} \label{eqn19}  \sigma^k_i(B\setminus A) = \sigma^k_i(\{-(i-1), -(i-2), \dots, -3, -1, 1\}).
    \end{equation}
    But, 
    \begin{equation} \label{eqn20}  \sigma^k_i(A)\setminus \sigma^k_i(B) = \{-C, C'\}
   \end{equation}
   and 
    \begin{equation} \label{eqn21}  \sigma^k_i(B)\setminus \sigma^k_i(A) =  \{-D, D'\}.
    \end{equation}    
    This leads to a contradiction for any $i< \ell$, since (by the equations \eqref{eqn18} and \eqref{eqn20}),
    $$|\sigma^k_i(A\setminus B)|=i-1=\ell-1.$$
    Also, by the equations     \eqref{eqn19} and \eqref{eqn21} ,
$$|\sigma^k_i(B\setminus A)|=i-1=\ell-1.$$
    This contradiction proves that the equation \eqref{eqn17} cannot be true. This proves the theorem.
     \end{proof}


\vskip6ex
  
  \begin{thebibliography}{9}  
  \bibitem[AH]{AH} M. \, Aschenbrenner and C. \, Hillar, {\it Finite generation of symmetric ideals}. {\it Transactions of A.M.S.} {\bf   359} (2007), 5171--5192.
  \bibitem[Di]{Di} R. Diestel, {\it Graph theory}, Springer-Verlag, New York,
1997.  

\bibitem[EH]{EH} D. Eisenbud and J. Harris, {\it The Geometry of Schemes}, GTM volume 197, Springer (2000).

\bibitem[GGW]{GGW} J. Geelen, B. Gerards and G. Whittle, Solving Rota's conjecture, {\it Notices of the AMS}  {\bf 61}, Number 7, (August 2014), 736--743.

\bibitem[Kur]{Kur} K. Kuratowski, Sur le probl\'eme des courbes gauches
en topologie, {\it Fund. Math.} {\bf  15} (1930), 271--283.
  
  
\bibitem[RS]{RS} N. Robertson and P. D. Seymour, Graph Minors. XX. Wagner’s Conjecture, {\it J.
Combin. Theory Ser. B} {\bf  92} (2004), 325--357.

\bibitem[Ro]{Ro} G.-C. Rota, Combinatorial theory, old and new, in:
{\it Proc. Internat. Cong. Math. (Nice, 1970)}, pp. 229--233.
Gauthier-Villars, Paris.  
  
  
  
  
  \end{thebibliography}
       \end{document}